\documentclass[12pt]{article}
\usepackage{amsmath,mathtools,amsfonts,amssymb,amsthm,bbm,color}
\usepackage[active]{srcltx}
\usepackage[utf8]{inputenc}

\DeclareMathOperator{\interior}{int}
\DeclareMathOperator{\ext}{ext}

\DeclareMathOperator{\dist}{dist}

\DeclareMathOperator{\argmin}{argmin}
\DeclareMathOperator{\argmax}{argmax}

\begin{document}
\newtheorem{oberklasse}{OberKlasse}
\newtheorem{lemma}[oberklasse]{Lemma}
\newtheorem{proposition}[oberklasse]{Proposition}
\newtheorem{theorem}[oberklasse]{Theorem}
\newtheorem{remark}[oberklasse]{Remark}
\newtheorem{corollary}[oberklasse]{Corollary}
\newtheorem{definition}[oberklasse]{Definition}
\newtheorem{assumption}{Assumption}

\newcommand{\R}{\mathbbm{R}}
\newcommand{\N}{\mathbbm{N}}
\newcommand{\Z}{\mathbbm{Z}}
\newcommand{\mc}{\mathcal}
\newcommand{\mf}{\mathfrak}
\newcommand{\eps}{\varepsilon}
\renewcommand{\phi}{\varphi}
\newcommand{\mynote}[1]{{\begin{center}\textcolor{red}{\textsc{#1}}\end{center}}}

\allowdisplaybreaks[1] 

\title{A linear programming approach to approximating the infinite time reachable set
of strictly stable linear control systems}
\author{Andreas Ernst, Lars Gr\"une, and Janosch Rieger}
\date{\today}
\maketitle

\begin{abstract}
We develop a new numerical method for approximating the infinite time 
reachable set of strictly stable linear control systems.
By solving a linear program with constraints that incorporate the system 
dynamics, we compute a polytope with fixed facet normals as an outer 
approximation of the limit set.
In particular, this approach does not rely on forward iteration
of finite-time reachable sets.
\end{abstract}

\medskip



\section{Introduction}

The approximation of finite time reachable sets of linear control systems 
has been studied by a number of mathematicians, engineers and computer scientists,
using a variety of approaches, so the following list is by no means exhaustive. 
In \cite{Baier}, a number of optimal control problems is solved to obtain 
a discretization of the support function of the reachable set in fixed directions,
while \cite{Graettinger} and \cite{Shao} essentially apply Benson's algorithm
to construct the reachable set adaptively.
Inner and outer approximations by zonotopes have been explored in \cite{Girard},
and inner and outer approximations by displacements of homothetic bodies
have been investigated in \cite{Fiacchini}.
Approximations by polytopes with fixed facet normals have been 
discussed in \cite{BenSassi}. 

\medskip 

For a strictly stable system with compact control input set, the limit of the finite
time reachable sets (as time tends to infinity) is a well-defined object, 
see \cite{Kolmanovsky}. 
We refer to it as the infinite time reachable set.
It is a compact subset of the state space, which is invariant under
the control dynamics and attracts any trajectory of the system.
The contracting dynamics allow the design of a priori and a posteriori estimates
for the Hausdorff error between an iterated compact set and the limit set,
see \cite{Artstein} and \cite{Rakovic05}.
As a consequence, all methods for the approximation of finite-time reachable sets
mentioned above can in principle be used to approximate the limit by forward
iteration.

\medskip

In the present paper, we use some of the above ideas, but pursue a completely
different approach to the approximation of the limit set, which is in a vague sense
conceptually similar with the papers \cite{Dorea}, \cite{Gaitsgory} 
and \cite{Korda} on weakly invariant sets of control systems.
In a first step, we analyze the induced dynamics on the space of the
nonempty compact sets, extending some of the results in \cite{Artstein}.
Then we use this insight and an approximation theorem from \cite{Rieger:17} 
to prove that an outer approximation by a polytope with fixed facet normals
can be computed as the solution to a disjunctive optimization problem,
avoiding the need for forward iteration.

\medskip

As the solution of a disjunctive program is difficult to compute, 
we construct a dual problem, which is a linear program and possesses 
the same global optimizer as the original disjunctive program.
In order to ensure uniqueness of the dual optimum, we inflate 
the system slightly to push the optimum away from some critical constraints
to create a situation in which this is relatively simple.
Then we use a stability result from \cite{Robinson} to conclude 
that the unperturbed dual problem has a unique solution, which coincides 
with the unique primal minimizer we wish to compute.

\section{Setting and Notation} \label{SN}

We fix a matrix $C\in\R^{d\times d}$ with spectral radius $\rho(C)<1$, 
a matrix $D\in\R^{d\times m}$ and a nonempty convex and compact control set 
$U\subset\R^m$, and we analyze the behavior of the control system
\begin{equation} \label{cs}
x_{k+1}=Cx_k+Du_k,\quad u_k\in U,
\end{equation}
on the unbounded time interval.
In particular, we will approximate its infinite time reachable set and
its all time reachable set by polytopes.

Throughout this paper, we fix a number $\ell\in(0,\rho(C))$.
By Lemma 5.6.10 in \cite{Horn:13}, there exists a norm 
$\|\cdot\|_\ell:\R^d\to\R_+$ such that the induced matrix norm 
satisfies $\|C\|_\ell\le\ell$.
Since $\R^d$ is finite-dimensional, there exist $c_{2,\ell},c_{\ell,2}>0$ with
\begin{equation}\label{equivalence}
c_{2,\ell}^{-1}\|x\|_\ell\le\|x\|
\le c_{\ell,2}\|x\|_\ell\quad\forall\,x\in\R^d,
\end{equation}
where $\|\cdot\|:\R^d\to\R_+$ denotes the Euclidean norm.

We denote the space of all nonempty compact subsets of $\R^d$ by $\mc{K}(\R^d)$
and the space of all nonempty compact and convex subsets of $\R^d$ 
by $\mc{K}_c(\R^d)$.
The Hausdorff semi-distance $\dist:\mc{K}(\R^d)\times\mc{K}(\R^d)\to\R_+$ 
and the corresponding symmetric Hausdorff distance 
$\dist_H:\mc{K}(\R^d)\times\mc{K}(\R^d)\to\R_+$ 
are defined by
\begin{align*}
&\dist(X,X')=\sup_{x\in X}\inf_{x'\in X'}\|x-x'\|,\\
&\dist_H(X,X')=\max\{\dist(X,X'),\dist(X',X)\}.
\end{align*}
For any $X\in\mc{K}(\R^d)$ and $R>0$, we write
\[B_R(X):=\{x\in\R^d:\dist(x,X)\le R\}\quad\text{and}\quad 
\|X\|:=\sup_{x\in X}\|x\|.\]
Identical notation with a subscript or superscript $\ell$ will be used
when the underlying norm is $\|\cdot\|_\ell:\R^d\to\R_+$.
Note that for any $X\in\mc{K}_c(\R^d)$ and $R>0$, the property
$B_R^\ell(X)\in\mc{K}_c(\R^d)$ still holds in this non-Euclidean
geometry by triangle inequality.

The support function of a set $X\in\mc{K}_c(\R^d)$ is a mapping
\[\sigma_X:\R^d\to\R,\quad\sigma_X(p):=\max_{x\in X}p^Tx.\]
For a set-valued map $F:\R^d\to\mc{K}(\R^d)$ and $X\in\mc{K}(\R^d)$,
we denote the image and the preimage of $X$ by
\[F(X):=\cup_{x\in X}F(x)\quad\text{and}\quad
F^{-1}(X):=\{x\in\R^d:F(x)\cap X\neq\emptyset\}.\]
The vector $\mathbbm{1}\in\R^N$ is the vector with the number $1$ 
in all $N$ components, and the vector $e_i$ is the $i$-th unit vector.
For any convex set $X\subset\R^d$, the set of extreme points of $X$
is denoted $\ext(X)$, and its interior is denoted $\interior(X)$.

\section{Preliminaries and auxiliary results} \label{PAR}

The fact that all norms on $\R^d$ are equivalent is reflected by a similar
statement for Hausdorff semi-distances and Hausdorff distances on $\mc{K}(\R^d)$.

\begin{lemma} \label{equivalent:Hausdorff}
The Hausdorff semi-distances $\dist$ and $\dist^\ell$ as well as the 
Hausdorff distances $\dist_H$ and $\dist_H^\ell$ are equivalent.
\end{lemma}

\begin{proof}
For any $X,X'\in\mc{K}(\R^d)$, we compute
\begin{align*}
&c_{2,\ell}^{-1}\dist^\ell(X,X')
=c_{2,\ell}^{-1}\sup_{x\in X}\inf_{x'\in X'}\|x-x'\|_\ell
\le \sup_{x\in X}\inf_{x'\in X'}\|x-x'\|
=\dist(X,X'),\\
&\dist(X,X')=\sup_{x\in X}\inf_{x'\in X'}\|x-x'\|
\le c_{\ell,2}\sup_{x\in X}\inf_{x'\in X'}\|x-x'\|_\ell
=c_{\ell,2}\dist^\ell(X,X'),
\end{align*}
and hence
\[c_{2,\ell}^{-1}\dist_H^\ell(X,X')\le\dist_H(X,X')\le c_{\ell,2}\dist_H^\ell(X,X').\]
\end{proof}

Given a matrix $A\in\R^{N\times d}$, we define a space of polyhedra 
by setting 
\[\mc{G}_A:=\{Q_{A,b}:b\in\R^N\}\setminus\{\emptyset\},\quad
Q_{A,b}:=\{x\in\R^d: Ax\le b\}.\]
This space has been explored in depth in the paper \cite{Rieger:17}.
We recapitulate the relevant facts as briefly as possible
and refer to \cite{Rieger:17} for technical details.

Throughout the rest of the paper, we require the following assumptions.

\begin{assumption} \label{A:ass}
The matrix $A\in\R^{N\times d}$ has the following properties.
\begin{itemize}
\item [a)] It consists of pairwise distinct rows $a_1^T,\ldots,a_N^T$ 
satisfying $a_i\in\R^d$ and $\|a_i\|_2=1$ for $i=1,\ldots,N$.
\item [b)] We have $Q_{A,0}=\{0\}$.
\end{itemize}
\end{assumption}

Assumption \ref{A:ass}b) holds whenever the rows of $A$ 
are reasonably dense in the sphere, see Theorem 16 in \cite{Rieger:17},
and by Corollary 17 in \cite{Rieger:17}, it
guarantees that the space $\mc{G}_A$ consists of (bounded) polytopes. 
By Theorem 13 in \cite{Rieger:17}, the mapping $b\mapsto Q_{A,b}$
is bi-Lipschitz w.r.t.\ Hausdorff distance.

\medskip

Intersections of polytopes can be expressed as the componentwise infimum
of their representations.

\begin{lemma} \label{intersections}
Let $\mc{B}\subset\R^N$ be a subset with $\cap_{b\in\mc{B}}Q_{A,b}\neq\emptyset$,
and let $b^*\in\R^N$ be given by $b^*_i:=\inf_{b\in\mc{B}}b_i$.
Then $Q_{A,b*}=\cap_{b\in\mc{B}}Q_{A,b}$.
\end{lemma}

\begin{proof}
If $x\in Q_{A,b^*}$, then $a_i^Tx\le b^*_i\le b_i$ for all $b\in\mc{B}$
and $i\in\{1,\ldots,N\}$, so $x\in\cap_{b\in\mc{B}}Q_{A,b}$.
If, on the other hand, we have $x\notin Q_{A,b^*}$, then there exists
$i\in\{1,\ldots,N\}$ with $b_i^*<a_i^Tx$.
By definition of the infimum, there exists $b'\in\mc{B}$ 
with $b_i^*\le b'_i<a_i^Tx$, and hence we have 
$x\notin Q_{A,b'}\supset(\cap_{b\in\mc{B}}Q_{A,b})$.
\end{proof}

The quantity 
\[\kappa_A:=\sup_{\|c\|_2=1}\,
\inf\Big\{\sum_{k=1}^Np_k\|a_k-\frac{1}{\|p\|_1}c\|_2:
p\in\ext(\{q\in\R^N:A^Tq=c,\ q\ge 0\})\Big\}\]
from Proposition 46 in \cite{Rieger:17} measures, roughly speaking, how easily
points on the unit sphere can be positively combined from the rows of $A$. 
It controls the approximation properties of the space $\mc{G}_A$
as a subspace of $\mc{K}_c(\R^d)$ in the following sense, 
see Theorem 47 in \cite{Rieger:17}.
\begin{theorem} \label{projector}
The mapping
\begin{align*}
&\pi_{\mc{G}_A}:\mc{K}_c(\R^d)\to\mc{G}_A,\quad 
\pi_{\mc{G}_A}(X):=Q_{A,b_X}\\
&(b_X)_i:=\max_{x\in X}a_i^Tx\quad\text{for}\quad i=1,\ldots,N,
\end{align*}
is a projector from $\mc{K}_c(\R^d)$ onto $\mc{G}_A$, it is
Lipschitz w.r.t.\ $\dist_H$, it is monotone w.r.t.\ inclusion, and it satisfies 
\[X\subset\pi_{\mc{G}_A}(X)\subset B(X,\kappa_A\|X\|)\quad
\forall\,X\in\mc{K}_c(\R^d).\]
\end{theorem}

Throughout this paper, we assume the following relation between the 
contraction rate $\ell$ of the matrix $C$ and the quality $\kappa_A$
of the approximation of $\mc{K}_c(\R^d)$ by $\mc{G_A}$.
It can be achieved by choosing $A$ such that its rows are sufficiently dense 
in the sphere, see the context of Proposition 46 in \cite{Rieger:17}.

\begin{assumption} \label{kappaineqass}
The matrix $A$ satisfies
\begin{equation*} 
\kappa_A<\frac{1-\ell}{c_{2,\ell}c_{\ell,2}\ell}.
\end{equation*}
\end{assumption}

We exploit this relation in the following lemma.

\begin{lemma} \label{not:worse:than:ell}
Let $X\in\mc{K}(\R^d)$, and define
\begin{equation*} %\label{Req}
R_A^X:=\frac{c_{2,\ell}c_{\ell,2}\ell\kappa_A}
{1-\ell-c_{2,\ell}c_{\ell,2}\ell\kappa_A}\|X\|_\ell.
\end{equation*}
For every $R\ge R_A^X$, we obtain %the estimate
%\[\dist^\ell(\pi_{\mc{G}_A}(B_R^\ell(X)),X)\le\ell^{-1}R\]
%and 
the inclusion
\[\pi_{\mc{G}_A}(B_R^\ell(X))
\subset B_{\ell^{-1}R}^\ell(X).\]
\end{lemma}

\begin{proof}
Since
\begin{align*}
c_{\ell,2}c_{2,\ell}\kappa_A\|X\|_\ell
=((1-\ell)\ell^{-1}-c_{\ell,2}c_{2,\ell}\kappa_A)R_A^X
\le((1-\ell)\ell^{-1}-c_{\ell,2}c_{2,\ell}\kappa_A)R,
\end{align*}
we can use Theorem \ref{projector} to compute
\begin{align*}
&\dist^\ell(\pi_{\mc{G}_A}(B_R^\ell(X)),B_R^\ell(X))
\le c_{\ell,2}\dist(\pi_{\mc{G}_A}(B_R^\ell(X)),B_R^\ell(X))\\
&\le c_{\ell,2}\kappa_A\|B_R^\ell(X)\|
\le c_{\ell,2}c_{2,\ell}\kappa_A\|B_R^\ell(X)\|_\ell\\
&\le c_{\ell,2}c_{2,\ell}\kappa_A(\|X\|_\ell+R)
\le(1-\ell)\ell^{-1}R,
\end{align*}
and we conclude that
\begin{align*}
&\dist^\ell(\pi_{\mc{G}_A}(B_R^\ell(X)),X)\\
&\le\dist^\ell(\pi_{\mc{G}_A}(B_R^\ell(X)),B_R^\ell(X))
+\dist^\ell(B_R^\ell(X),X)\\
&\le(1-\ell)\ell^{-1}R+R=\ell^{-1}R,
\end{align*}
which implies the desired inclusion.
\end{proof}

\section{Properties of the infinite time reachable set} \label{ISP}

Let $V\in\mc{K}_c(\R^d)$, and consider the mapping
\[F:\R^d\to\mc{K}_c(\R^d),\quad F(x):=Cx+V.\]
This setting includes system \eqref{cs} for $V=DU$, and it allows us 
to treat $\eps$-inflations of $F$ with minimal notational complication.

As we need precise statements in the norms we work with, we prove a few 
facts that may in principle be well-known. 
Part c) of the following proposition is, e.g.,
similar to Proposition 4.3 in \cite{Artstein}.

\begin{proposition} \label{inf:reach}
The following statements hold.
\begin{itemize}
\item [a)] The map $F:\R^d\to\mc{K}_c(\R^d)$ is $\ell$-Lipschitz w.r.t.\
$\|\cdot\|_\ell:\R^d\to\R_+$, i.e.
\[\dist_H^\ell(F(x),F(x'))\le\ell\|x-x'\|_\ell\quad\forall\,x,x'\in\R^d.\]
\item [b)] The mapping 
$\mc{F}:\mc{K}(\R^d)\to\mc{K}(\R^d)$ given by $\mc{F}(X):=F(X)$
maps $\mc{K}_c(\R^d)$ into itself and satisfies
\[\dist_H^\ell(\mc{F}(X),\mc{F}(X'))
\le\ell\dist_H^\ell(X,X')\quad\forall\,X,X'\in\mc{K}(\R^d).\]
\item [c)] There exists a unique set $X^*\in\mc{K}(\R^d)$ with $X^*=F(X^*)$,
and for any $X_0\in\mc{K}(\R^d)$, the sequence 
$\{X_k\}_{k\in\N}\subset\mc{K}(\R^d)$ given by
$X_{k+1}=F(X_k)$ for all $k\in\N$ satisfies
\begin{align*} 
&\lim_{k\to\infty}\dist_H^\ell(X_k,X^*)=0,\\
&\dist_H^\ell(X_{k+1},X^*)\le\ell\dist_H^\ell(X_k,X^*)\quad\forall\,k\in\N,\\
&\dist_H^\ell(X_k,X^*)\le\tfrac{\ell^k}{1-\ell}\dist_H^\ell(X_1,X_0).
\end{align*}
\item [d)] We have $X^*\in\mc{K}_c(\R^d)$, and if $X_0\in\mc{K}_c(\R^d)$, then
the sequence $\{X_k\}_{k\in\N}$ in part c) satisfies 
$\{X_k\}_{k\in\N}\subset\mc{K}_c(\R^d)$.
\item [e)] If $X\in\mc{K}(\R^d)$ satisfies $F(X)\subset X$, then $X^*\subset X$.
Conversely, if $X\subset F(X$), then $X\subset X^*$.
\item [f)] We have the a priori estimate $\|X^*\|_\ell\le\frac{1}{1-\ell}\|V\|_\ell$.
\end{itemize}
\end{proposition}

\begin{proof}
a) For all $x,x'\in\R^d$, we compute
\[\dist^\ell(F(x),F(x'))
=\sup_{v\in V}\inf_{v'\in V}\|Cx+v-Cx'-v'\|_\ell
\le\ell\|x-x'\|_\ell.\]

b) For any $X\in\mc{K}(\R^d)$, we have $\mc{F}(X)=F(X)\in\mc{K}(\R^d)$ 
by Corollary 2.20 and Theorem 2.68 in \cite{Hu:97}.
If $X\in\mc{K}_c(\R^d)$, then $\mc{F}(X)=F(X)=CX+V$ is a Minkowski sum of
two convex sets, and hence $\mc{F}(X)\in\mc{K}_c(\R^d)$,
see \cite[page 47]{Schneider}.  
For any $X,X'\in\mc{K}(\R^d)$, we compute
\begin{align*}
&\dist^\ell(\mc{F}(X),\mc{F}(X'))
=\sup_{x\in X,\ v\in V}\inf_{x'\in X',\ v'\in V}\|Cx+v-Cx'-v'\|_\ell\\
&\le\sup_{x\in X}\inf_{x'\in X'}\|Cx-Cx'\|_\ell
\le\ell\dist^\ell(X,X').
\end{align*}

c) By Proposition 1.6 in \cite{Hu:97}, the space $(\mc{K}(\R^d),\dist_H^\ell)$ 
is complete.
Applying the contraction mapping principle (Theorem 1.A in \cite{Zeidler:86})
to $\mc{F}$ in this situation yields the desired statement.

d) If $X_0\in\mc{K}_c(\R^d)$, then $\{X_k\}_{k\in\N}\subset\mc{K}_c(\R^d)$
follows by induction from part b).
The choice $X_0=\{0\}$ induces a sequence $\{X_k\}_{k\in\N}\subset\mc{K}_c(\R^d)$, 
which, by part c), satisfies $\lim_{k\to\infty}\dist_H^\ell(X_k,X^*)=0$.
According to Theorem 1.8.3 in \cite{Schneider}, % and Lemma \ref{equivalent:Hausdorff},
the space $(\mc{K}_c(\R^d),\dist_H^\ell)$ is complete, so $X^*\in\mc{K}_c(\R^d)$.

e) If $F(X)\subset X$, then by induction, we have $F^{k+1}(X)\subset F^k(X)$ 
for all $k\in\N$, so the desired statement follows from part c).
The proof of the opposite inclusion is analogous.

f) We use part c) to compute
\begin{align*}
&\dist^\ell(X^*,0)
\le\dist^\ell(X^*,F(0))+\dist^\ell(F(0),0)\\
&=\dist^\ell(F(X^*),F(0))+\dist^\ell(V,0)
\le\ell\dist^\ell(X^*,0)+\|V\|_\ell,
\end{align*}
subtract $\ell\dist^\ell(X^*,0)$ and divide by $1-\ell$.
\end{proof}

\section{Approximation of $X^*$ via minimization} \label{AVOP}

For technical reasons, we will have to consider the situation when
$F$ is inflated by an $\eps$-ball.
This yields a perturbed mapping
\[F_\eps:\R^d\to\mc{K}_c(\R^d),\quad F_\eps(x):=Cx+B_\eps(V).\]

The following auxiliary result discusses the interplay between contraction, 
nested dynamics and overapproximation.

\begin{lemma} \label{nested:lemma}
Let $X\in\mc{K}(\R^d)$ with $F(X)\subset X$, and let $R>0$.
Then we have $F(B^\ell_{\ell^{-1}R}(X))\subset B^\ell_R(X)$.
The same statement holds for the mapping $F_\eps$.
\end{lemma}

\begin{proof}
This follows from the computation
\begin{align*}
&\dist^\ell(F(B^\ell_{\ell^{-1}R}(X)),X)
\le\dist^\ell(F(B^\ell_{\ell^{-1}R}(X)),F(X))\\
&\le\ell\dist^\ell(B^\ell_{\ell^{-1}R}(X),X)
\le R.
\end{align*}
Since $F_\eps$ shares all properties of $F$, the same proof applies to $F_\eps$.
\end{proof}

We construct an approximation to $X^*$ by solving an optimization problem, 
which uses the property proved in Proposition \ref{inf:reach} part e) 
as a constraint.

\begin{proposition}\label{vacuum:approximation}
The optimization problem
\begin{equation}\left.\begin{aligned} 
&\min_b\,\mathbbm{1}^Tb\quad\text{subject to}\quad b\in\mc{B}_{X^*}\\
&\mc{B}_{X^*}:=\{b\in\R^N:F(Q_{A,b})\subset Q_{A,b}\neq\emptyset\}
\end{aligned}\right\}\label{op1}\end{equation}
possesses a unique solution $b^*\in\R^N$. 
This $b^*$ is given by 
\[b^*_i=\inf_{b\in\mc{B}_{X^*}}b_i\quad\text{for}\quad i\in\{1,\ldots,N\}\]
and satisfies the error bound
\[X^*\subset Q_{A,b^*}\subset B^\ell_{\ell^{-1}R_A^{X^*}}(X^*)\]
with $R_A^{X^*}$ as in Lemma \ref{not:worse:than:ell}.
\end{proposition}

\begin{remark}
a) Note that Problem \eqref{op1} selects the smallest polytope
in the collection $\{Q_{A,b}:b\in\mc{B}_{X^*}\}$ with respect to inclusion.
The setup of the problem also guarantees that the vector $b^*$ is a particularly 
nice representation of the polytope $Q_{A,b^*}$, see Section 2.2 
of \cite{Rieger:17}.

b) The number $R_A^{X^*}$ is defined in Lemma \ref{not:worse:than:ell} and
bounded by the a-priori estimate in Proposition \ref{inf:reach} part f).

c) It is at this stage not obvious that Problem \eqref{op1} is a
disjunctive program.
This will be established in the next section.
\end{remark}

\begin{proof}[Proof of Proposition \ref{vacuum:approximation}]
We clearly have $\pi_{\mc{G}_A}(B_{R_A^{X^*}}^\ell(X^*))\neq\emptyset$.
Lemma \ref{not:worse:than:ell} implies
\[\pi_{\mc{G}_A}(B_{R_A^{X^*}}^\ell(X^*))\subset B^\ell_{\ell^{-1}R_A^{X^*}}(X^*),\]
and by Proposition \ref{inf:reach} part c), Lemma \ref{nested:lemma} 
and Theorem \ref{projector}, we have
\[F(\pi_{\mc{G}_A}(B_{R_A^{X^*}}^\ell(X^*)))
\subset F(B^\ell_{\ell^{-1}{R_A^{X^*}}}(X^*))
\subset B_{R_A^{X^*}}^\ell(X^*)
\subset\pi_{\mc{G}_A}(B_{R_A^{X^*}}^\ell(X^*)).\]
In particular, we find
\[\pi_{\mc{G}_A}(B_{R_A^{X^*}}^\ell(X^*))\in\{Q_{A,b}:b\in\mc{B}_{X^*}\},\]
so $\mc{B}_{X^*}\neq\emptyset$.
According to Proposition \ref{inf:reach} part e), we have
$X^*\subset(\cap_{b\in\mc{B}_{X^*}}Q_{A,b})$,
so by Lemma \ref{intersections}, the vector $b^*\in\R^N$ given by 
$b^*_i:=\inf_{b\in\mc{B}}b_i$ satisfies
\[X^*\subset(\cap_{b\in\mc{B}_{X^*}}Q_{A,b})=Q_{A,b^*}.\]
From the definition of $\mc{B}_{X^*}$, we obtain
\begin{align*}
&F(Q_{A,b^*})
=F(\cap_{b\in\mc{B}_{X^*}}Q_{A,b})
=C(\cap_{b\in\mc{B}_{X^*}}Q_{A,b})+V
\subset(\cap_{b\in\mc{B}_{X^*}}(CQ_{A,b})+V)\\
&\subset(\cap_{b\in\mc{B}_{X^*}}(CQ_{A,b}+V))
=\cap_{b\in\mc{B}_{X^*}}F(Q_{A,b})
\subset(\cap_{b\in\mc{B}_{X^*}} Q_{A,b})
=Q_{A,b^*},
\end{align*}
so we have $b^*\in\mc{B}_{X^*}$ as well.
By the above and by Lemma \ref{not:worse:than:ell}, we conclude 
$b^*=\argmin_{b\in\mc{B}_{X^*}}\mathbbm{1}^Tb$ and
\[X^*\subset Q_{A,b^*}\subset\pi_{\mc{G}_A}(B_{R_A^{X^*}}^\ell(X^*))
\subset B^\ell_{\ell^{-1}R_A^{X^*}}(X^*).\]
\end{proof}

The unique minimizers of the perturbed problems approximate
the unique minimizer of the original problem.

\begin{proposition}\label{eps:converge}
For any $\eps>0$, the optimization problem
\begin{equation}\left.\begin{aligned}
&\min_b\,\mathbbm{1}^Tb\quad\text{subject to}\quad b\in\mc{B}_{X^*_\eps}\\
&\mc{B}_{X^*_\eps}:=\{b\in\R^N:F_\eps(Q_{A,b})\subset Q_{A,b}\neq\emptyset\}
\end{aligned}\right\}\label{op1eps}\end{equation}
possesses a unique solution $b^*_\eps\in\R^N$.
This $b^*_\eps$ is given by 
\[b^*_{\eps,i}=\inf_{b\in\mc{B}_{X^*_\eps}}
b_i\quad\text{for}\quad i\in\{1,\ldots,N\},\]
and we have
\[\lim_{\eps\searrow 0}b_\eps^*=b^*.\]
\end{proposition}

\begin{proof}
Since $F_\eps$ shares all properties of $F$ for every $\eps>0$, we can apply 
Proposition \ref{vacuum:approximation} to the mapping $F_\eps$
to obtain existence and uniqueness of the solutions $b^*_\eps$.
It remains to show the convergence statement.

The inclusion
$F(Q_{A,b_\eps^*})\subset F_\eps(Q_{A,b_\eps^*})\subset Q_{A,b_\eps^*}$
implies $b_\eps^*\in\mc{B}_{X^*}$, so $b^*\le b_\eps^*$ holds by 
Proposition \ref{vacuum:approximation}, and hence 
$Q_{A,b^*}\subset Q_{A,b_\eps^*}$.
Let $R_1=R_1(\eps):=\frac{c_{2,\ell}\ell\eps}{1-\ell}$ and define
$X(\eps):=B^\ell_{\ell^{-1}R_1}(Q_{A,b^*})$. 
Lemma \ref{nested:lemma} gives
\begin{align*}
&F_\eps(X(\eps))=F_\eps(B^\ell_{\ell^{-1}R_1}(Q_{A,b^*}))
=F(B^\ell_{\ell^{-1}R_1}(Q_{A,b^*}))+B_\eps(0)\\
&\subset B_{R_1}^\ell(Q_{A,b^*})+B_{c_{2,\ell}\eps}^\ell(0)
=B_{R_1+c_{2,\ell}\eps}^\ell(Q_{A,b^*})
\subset B^\ell_{\ell^{-1}R_1}(Q_{A,b^*})
=X(\eps).
\end{align*}
Now let $R_2=R_2(\eps):=R_A^{X(\eps)}$ with notation as in 
Lemma \ref{not:worse:than:ell}.
Using Lemma \ref{not:worse:than:ell}, Lemma \ref{nested:lemma}
and Theorem \ref{projector}, we obtain
\begin{align*}
F_\eps(\pi_{\mc{G}_A}(B_{R_2}^\ell(X(\eps))))
\subset F_\eps(B_{\ell^{-1}R_2}^\ell(X(\eps)))
\subset B_{R_2}^\ell(X(\eps))
\subset\pi_{\mc{G}_A}(B_{R_2}^\ell(X(\eps))),
\end{align*}
so $\pi_{\mc{G}_A}(B_{R_2}^\ell(X(\eps)))\in\{Q_{A,b}:b\in\mc{B}_{X^*_\eps}\}$,
and by minimality of $b_\eps^*$, we have
\[Q_{A,b^*}\subset Q_{A,b^*_\eps}\subset\pi_{\mc{G}_A}(B_{R_2}^\ell(X(\eps))).\]
Let $L_A>0$ be the Lipschitz constant of the mapping $\pi_{\mc{G}_A}$.
By the above, and since $Q_{A,b^*}\in\mc{G}_A$, we obtain
\begin{align*}
&\dist^\ell(Q_{A,b^*_\eps},Q_{A,b^*})
\le\dist^\ell(\pi_{\mc{G}_A}(B_{R_2}^\ell(X(\eps))),Q_{A,b^*})\\
&=\dist^\ell(\pi_{\mc{G}_A}(B_{R_2}^\ell(X(\eps))),\pi_{\mc{G}_A}(Q_{A,b^*}))
\le L_A\dist^\ell(B_{R_2}^\ell(X(\eps)),Q_{A,b^*}),
\end{align*}
which implies
\begin{align*}
&\lim_{\eps\searrow 0}\dist^\ell(Q_{A,b^*_\eps},Q_{A,b^*})=0
\end{align*}
and hence the desired convergence statement.
\end{proof}

\section{Disjunctive programs} \label{CC}

We assume that the values $\sigma_V(a_i)$ for $i\in\{1,\ldots,N\}$ 
of the support function of the sets $V$ are available.
This is not a strong requirement, because in many applications, 
the set $V$ has a very simple shape.
We use the notation
\begin{align*}
&P_0:=\{p\in\R^N:A^Tp=0,\ \mathbbm{1}^Tp=1,\ p\ge 0\},\\
&P_i:=\{p\in\R^N,\ A^Tp=C^Ta_i,\ p\ge 0\},\quad i\in\{1,\ldots,N\},
\end{align*}
and we develop representation of the sets $\mc{B}_{X^*_\eps}$
which is accessible to linear optimization techniques.
If $\eps=0$, then $F_\eps=F$, $B_\eps(V)=V$ and $\mc{B}_{X^*_\eps}=\mc{B}_{X^*}$.

\begin{proposition} \label{lpc}
Consider an arbitrary vector $b\in\R^N$ and $\eps\ge 0$.
\begin{itemize}
\item [a)] For $X\in\mc{K}(\R^d)$, the inclusion 
$X\subset Q_{A,b}$ holds if and only if we have $\sigma_{X}(a_i)\le b_i$ 
for all $i\in\{1,\ldots,N\}$.
\item [b)] The following statements are equivalent:
\begin{itemize}
\item [i)] $Q_{A,b}\neq\emptyset$;
\item [ii)] $p^Tb\ge 0$ for all $p\in\R^N$ with $A^Tp=0$ and $p\ge 0$;
\item [iii)] $p^Tb\ge 0$ for all $p\in P_0$.
\item [iv)] $p^Tb\ge 0$ for all $p\in\ext(P_0)$.
\end{itemize}
\item [c)] If we have $Q_{A,b}\neq\emptyset$, then the following statements
are equivalent:
\begin{itemize}
\item [i)] $F_\eps(Q_{A,b})\subset Q_{A,b}$;
\item [ii)] $\max\{(C^Ta_i)^Tx:x\in Q_{A,b}\}\le b_i-\sigma_{B_\eps(V)}(a_i)$
$\forall\,i\in\{1,\ldots,N\}$;
\item [iii)] $\min\{(p-e_i)^Tb:p\in P_i\}\le-\sigma_{B_\eps(V)}(a_i)$ 
$\forall\,i\in\{1,\ldots,N\}$;
\item [iv)] $\min\{(p-e_i)^Tb:p\in\ext(P_i)\}\le-\sigma_{B_\eps(V)}(a_i)$
$\forall\,i\in\{1,\ldots,N\}$.
\end{itemize}
\end{itemize}
\end{proposition}

\begin{proof}
Statement a) is obvious.

b) The equivalence between i) and ii) is the version of the Farkas lemma 
given in Proposition 1.7 of \cite{Ziegler}.
Elementary arguments show that statement ii) is equivalent with statement iii).
Since the set $P_0$ is a compact polytope, statement iii) is equivalent 
with statement iv).

c) We have $F_\eps(Q_{A,b})\subset Q_{A,b}$ if and only if 
\[a_i^T(Cx+v)\le b_i\quad\forall\,x\in Q_{A,b},\ \forall\,v\in B_\eps(V),\
\forall\,i\in\{1,\ldots,N\},\]
which can be rewritten as
\[(C^Ta_i)^Tx\le b_i-a_i^Tv\quad\forall\,x\in Q_{A,b},\ \forall\,v\in B_\eps(V),\
\forall\,i\in\{1,\ldots,N\}.\]
This establishes the equivalence of statements i) and ii).
Since $Q_{A,b}$ is nonempty and bounded, the strong duality theorem 
for linear programming as presented in Theorem 4.13 of \cite{Theobald:13}
guarantees that 
\begin{align*}
&\max\{(C^Ta_i)^Tx:x\in Q_{A,b}\}=\min\{p^Tb:p\in P_i\}
=\min\{p^Tb:p\in\ext(P_i)\}
\end{align*}
is finite.
Hence statement ii) is equivalent with statements iii) and iv).
\end{proof}

We state Problems \eqref{op1} and \eqref{op1eps} as a disjunctive programs 
to highlight their structural properties.

\begin{theorem}
Problems \eqref{op1} and \eqref{op1eps} are equivalent with the problem
\begin{equation}\label{LP:ext:1}
\left.\begin{aligned}
&\min_b\,\mathbbm{1}^Tb&&\\
&\text{subject to}&0&\le p^Tb\quad\forall\,p\in P_0\\
&&\min\{(p-e_1)^Tb:p\in\ext(P_1)\}&\le-\sigma_{B_\eps(V)}(a_1),\\
&&&\vdotswithin{\le}\\
&&\min\{(p-e_N)^Tb:p\in\ext(P_N)\}&\le-\sigma_{B_\eps(V)}(a_N)
\end{aligned}\right\}
\end{equation}
with $\eps=0$ in the case of Problem \eqref{op1}.
\end{theorem}

Disjunctive programs are, in general, hard to solve, see \cite{Balas}.
To compute the desired solution $b^*$, we will construct a dual-type
problem that is just an ordinary linear program.

\section{Perturbed dual LPs}

In the following, we will formulate a linear program, which is related
to the dual of Problem \eqref{LP:ext:1} with $\eps>0$ in the sense of \cite{Balas}.

\begin{remark}
It is, in general, not possible to compute the solution $b^*$ of 
Problem \ref{LP:ext:1} with $\eps=0$ by following \cite{Balas} directly:

a) If $\interior(X^*)\neq\emptyset$, then $\interior(Q_{A,b^*})\neq\emptyset$,
which implies $p^Tb^*>0$ for all $p\in P_0$ by Proposition 36 in \cite{Rieger:17}.
The dual as defined in \cite{Balas} and similar works involves a constraint
of type $p^Tb\le 0$ for some $p\in P_0$, which means that $b^*$ is dual infeasible 
in this setting.

b) The dual problem from \cite{Balas} may have more than one maximizer, 
so it is not obvious how to recover $b^*$ from a dual solution.
\end{remark}

These facts motivate us to construct a dual problem following the general idea from
\cite{Balas}, but omitting the constraints $p^Tb\ge 0$ for all $p\in P_0$, and to
use a perturbation argument to show uniqueness of the dual maximizer.

\begin{proposition}\label{locmin}
For any $\eps>0$, the global minimizer $b^*_\eps\in\R^N$ 
of Problem \eqref{op1eps} satisfies
\begin{align}
&\max\{(C^Ta_i)^Tx:x\in Q_{A,b^*_\eps}\}=b_{\eps,i}^*-\sigma_{B_\eps(V)}(a_i)\quad
\forall\,i\in\{1,\ldots,N\},\label{primal:eq}\\
&\min\{(p-e_i)^Tb^*_\eps:p\in\ext(P_i)\}=-\sigma_{B_\eps(V)}(a_i)\quad\forall\,i\in\{1,\ldots,N\},
\label{dual:eq}
\end{align}
and the global minimizer $b^*\in\R^N$ of Problem \eqref{op1} satisfies
\[\min\{(p-e_i)^Tb^*:p\in\ext(P_i)\}=-\sigma_{V}(a_i)\quad\forall\,i\in\{1,\ldots,N\}.\]
In particular, we have $b^*_\eps\in\Omega_\eps$ and $b^*\in\Omega$, where
\begin{align*}
&\Omega_\eps:=\{b\in\R^N:(e_i-p)^Tb\le\sigma_{B_\eps(V)}(a_i)\ \forall\,p\in\ext(P_i),\ i\in\{1,\ldots,N\}\},\\
&\Omega:=\{b\in\R^N:(e_i-p)^Tb\le\sigma_{V}(a_i)\ \forall\,p\in\ext(P_i),\ i\in\{1,\ldots,N\}\}.
\end{align*}
\end{proposition}

\begin{proof}
Since $b^*_\eps\in\mc{B}_{X^*_\eps}$, 
we have $Q_{A,b^*_\eps}\neq\emptyset$, 
and hence $F_\eps(Q_{A,b^*_\eps})\neq\emptyset$, 
as well as $F_\eps(Q_{A,b^*_\eps})\subset Q_{A,b^*_\eps}$, which is,
by Proposition \ref{lpc} part c), equivalent with
\begin{equation}\label{forall}
\max\{(C^Ta_i)^Tx:x\in Q_{A,b^*_\eps}\}\le b^*_{\eps,i}-\sigma_{B_\eps(V)}(a_i)\quad
\forall\,i\in\{1,\ldots,N\}.
\end{equation}
Assume that there exist $j\in\{1,\ldots,N\}$ and $\delta>0$ with
\begin{equation}\label{forj}
\max\{(C^Ta_j)^Tx:x\in Q_{A,b^*_\eps}\}\le b^*_{\eps,j}-\sigma_{B_\eps(V)}(a_j)-\delta.
\end{equation}
Then for $b^\delta:=b^*_\eps-\delta e_j$, inequalities \eqref{forall} 
and \eqref{forj} yield
\begin{align*}
\sigma_{F_\eps(Q_{A,b^*_\eps})}(a_i)
&=\max\{(C^Ta_i)^Tx:x\in Q_{A,b^*_\eps}\}+\sigma_{B_\eps(V)}(a_i)\\
&\le b^\delta_i\quad\forall\,i\in\{1,\ldots,N\},
\end{align*}
so $\emptyset\neq F_\eps(Q_{A,b^*_\eps})\subset Q_{A,b^\delta}$ follows from
Proposition \ref{lpc} part a).
By monotonicity, we have
\[\sigma_{F_\eps(Q_{A,b^\delta})}(a_i)
\le\sigma_{F_\eps(Q_{A,b^*_\eps})}(a_i)
\le b_i^\delta\quad\forall\,i\in\{1,\ldots,N\},\]
which is equivalent with $F_\eps(Q_{A,b^\delta})\subset Q_{A,b^\delta}$.
Hence $b^\delta\in\mc{B}_{X^*_\eps}$, but we have 
$\mathbbm{1}^Tb^\delta<\mathbbm{1}^Tb^*_\eps$,
which is a contradiction.
All in all, we have proved equation \eqref{primal:eq},
and equation \eqref{dual:eq} follows from the strong duality theorem 
of linear programming.
Equation \eqref{dual:eq}, in turn, implies that
\[(p-e_i)^Tb^*_\eps\ge-\sigma_{B_\eps(V)}(a_i)\quad\forall\,p\in\ext(P_i),\ 
\forall\,i\in\{1,\ldots,N\},\]
which shows that $b^*_\eps\in\Omega_\eps$.
The same arguments work in the case $\eps=0$.
\end{proof}

The following result shows that the set $Q_{A,b^*_\eps}$ can be computed 
by solving a linear programming problem for $b^*_\eps$.

\begin{proposition} \label{perturbed:unique}
For any $\eps>0$, the unique solution $b^*_\eps$ of the perturbed disjunctive
program \eqref{op1eps} is the unique solution of the linear program
\[\max_b\,\mathbbm{1}^Tb\quad\text{subject to}\quad b\in\Omega_\eps.\]
\end{proposition}

The need to consider an arbitrarily small inflation of the set $V$ 
arises from the following proof, in which we need that small perturbations $b$
of the point $b^*_\eps$ satisfy $Q_{A,b}\neq\emptyset$.

\begin{proof}
Assume that there exists $b_*\in\Omega_\eps\setminus\{b^*_\eps\}$ with 
$\mathbbm{1}^Tb_*\ge\mathbbm{1}^Tb^*_\eps$. 
Since $b^*_\eps\in\mc{B}_{X^*_\eps}$, we have
\[\emptyset\neq\interior B_\eps(V)
\subset\interior F_\eps(Q_{A,b^*_\eps})\subset\interior Q_{A,b^*_\eps},\]
and by Proposition 37 in \cite{Rieger:17}, there exists 
$\delta>0$ such that $p^Tb^*_\eps\ge\delta$ for all $p\in\ext(P_0)$. 
By H\"older inequality, the vector 
$\bar{b}:=b^*_\eps+\tfrac{\delta}{\|b^*_\eps-b_*\|_\infty}(b^*_\eps-b_*)$
satisfies
\[p^T\bar{b}
=p^Tb^*_\eps+\tfrac{\delta}{\|b^*_\eps-b_*\|_\infty}p^T(b^*_\eps-b_*)
\ge\delta-\delta\|p\|_1
=0\quad\forall p\in\ext(P_0),\]
so $Q_{A,\bar{b}}\neq\emptyset$ by part b) of Proposition \ref{lpc}.
Since $b_*\in\Omega_\eps$ and by Proposition \ref{locmin}, for every $i\in\{1,\ldots,N\}$,
there exists $p\in\ext(P_i)$ such that
\[(e_i-p)^Tb_*\le\sigma_{B_\eps(V)}(a_i)\quad\text{and}\quad
(p-e_i)^Tb^*_\eps=-\sigma_{B_\eps(V)}(a_i).\]
From this we conclude that
\begin{align*}
(p-e_i)^T\bar{b}
&=(p-e_i)^T(b^*_\eps+\tfrac{\delta}{\|b^*_\eps-b_*\|_\infty}(b^*_\eps-b_*))\\
&=(1+\tfrac{\delta}{\|b^*_\eps-b_*\|_\infty})(p-e_i)^Tb^*_\eps-\tfrac{\delta}{\|b^*_\eps-b_*\|_\infty}(p-e_i)^Tb_*\\
&\le-(1+\tfrac{\delta}{\|b^*_\eps-b_*\|_\infty})\sigma_{B_\eps(V)}(a_i)
+\tfrac{\delta}{\|b^*_\eps-b_*\|_\infty}\sigma_{B_\eps(V)}(a_i)
=-\sigma_{B_\eps(V)}(a_i),
\end{align*}
and hence that $F_\eps(Q_{A,\bar{b}})\subset Q_{A,\bar{b}}$ according to
Proposition \ref{lpc} part c).
All in all, we have $\bar{b}\in\mc{B}_{X^*_\eps}$, but 
\[\mathbbm{1}^T\bar{b}
=\mathbbm{1}^T(b^*_\eps+\tfrac{\delta}{\|b^*_\eps-b_*\|_\infty}(b^*_\eps-b_*))
=(1+\tfrac{\delta}{\|b^*_\eps-b_*\|_\infty})\mathbbm{1}^Tb^*_\eps
-\tfrac{\delta}{\|b^*_\eps-b_*\|_\infty}\mathbbm{1}^Tb_*
\le\mathbbm{1}^Tb^*_\eps,\]
which is impossible, because the point $b^*_\eps$ is the unique global minimum 
of Problem \eqref{op1}.
\end{proof}

\section{The unperturbed dual LP}

Now we conclude that the approximation $Q_{A,b^*}$ to $X^*$ we wish to compute
is indeed given by the unique solution of the unperturbed dual linear program.

\begin{theorem}
The unique solution $b^*$ of the disjunctive program \eqref{op1} 
is the unique solution of the linear program
\[\max_b\,\mathbbm{1}^Tb\quad\text{subject to}\quad b\in\Omega.\]
\end{theorem}

\begin{proof}
By Proposition \ref{perturbed:unique}, for any $\eps>0$, the unique solution 
$b^*_\eps$ of the disjunctive program \eqref{op1eps} is the unique solution 
of the linear program
\begin{equation} \label{loc:1}
\max_b\,\mathbbm{1}^Tb\quad\text{subject to}\quad b\in\Omega_\eps.
\end{equation}
By Proposition \ref{locmin}, we have $b^*\in\Omega$, so the linear program
\begin{equation} \label{loc:2}
\max_b\,\mathbbm{1}^Tb\quad\text{subject to}\quad b\in\Omega
\end{equation}
is feasible. 
Since $\Omega\subset\Omega_\eps$, the value of Problem \eqref{loc:2} is bounded 
by the value of Problem \eqref{loc:1} with $\eps=1$, so 
$\argmax_{b\in\Omega}\mathbbm{1}^Tb\neq\emptyset$.
Corollary 3.1 from \cite{Robinson} yields that if 
$\tilde{b}\in\argmax_{b\in\Omega}\mathbbm{1}^Tb$ and there exists $\eps_0>0$ with
$\argmax_{b\in\Omega_\eps}\mathbbm{1}^Tb\neq\emptyset$ for $\eps\in(0,\eps_0]$, then
there exist $\tilde{b}_\eps\in\argmax_{b\in\Omega_\eps}\mathbbm{1}^Tb$ 
with
\[\tilde{b}=\lim_{\eps\searrow 0}\tilde{b}_\eps.\]
In the present situation, we have $b^*_\eps=\argmax_{b\in\Omega_\eps}\mathbbm{1}^Tb$,
so using Proposition \ref{eps:converge}, we conclude that
\[\argmax_{b\in\Omega}\mathbbm{1}^Tb=\lim_{\eps\searrow 0}b^*_\eps=b^*.\]
\end{proof}

\subsection*{Acknowledgement}
We thank Andrew Eberhard for pointing us towards the term \emph{disjunctive
programming}, which we would otherwise never have found in the maze of mathematical
literature.

\bibliographystyle{plain}
\bibliography{appinf}

\end{document}